\documentclass[12pt]{article}

\usepackage[T1]{fontenc}
\usepackage[utf8]{inputenc}
\usepackage{authblk}
\usepackage{amsmath, amssymb,amsthm}

\usepackage{geometry}

\usepackage[all]{xypic}

\usepackage{caption}
\usepackage{graphicx}
\usepackage{indentfirst}

\usepackage{color}
\usepackage{comment}
\usepackage{xpatch}

\usepackage[colorlinks,
linkcolor=blue,
anchorcolor=blue,
citecolor=red
]{hyperref}

\newtheorem{theorem}{Theorem}

\newtheorem{claim}{Claim}

\newtheorem{lemma}{Lemma}

\newtheorem{remark}{Remark}
\newtheorem{sublemma}{Sublemma}
\newtheorem{prop}{Proposition}

\newtheorem{definition}{Definition}

\usepackage{galois}
\numberwithin{equation}{section}

\begin{document}
\captionsetup[figure]{labelfont={bf},name={},labelsep=period}
\pagestyle{plain}

\title{Maximal first Betti number rigidity of noncompact \texttt{RCD}(0,$N$) spaces}
\author{Zhu Ye\thanks{Supported partially by  National Natural Science Foundation of China [11821101] and  Beijing Natural Science Foundation [Z19003].\\	Email address: 2210501006@cnu.edu.cn.}}
\affil{Department of Mathematics, Capital Normal University}
\date{}
\maketitle
\begin{abstract}
Let $(M,d,\mathfrak{m})$ be a noncompact \texttt{RCD}(0,$N$) space with $N\in\mathbb{N}_+$ and $\text{supp}\mathfrak{m}=M$. We prove that if the  first Betti number of $M$ equals $N-1$, then $(M,d,\mathfrak{m})$ is either a flat Riemannian $N$-manifold with a soul $T^{N-1}$  or the metric product $[0,\infty)\times T^{N-1}$, both with  the measure a multiple of the Riemannian volume, where $T^{N-1}$ is a flat torus. 
\end{abstract}

\section{Introduction}
Let $K\in\mathbb{R}$ and $N\in[1,\infty)$ throughtout the paper. The $\texttt{RCD}(K,N)$ spaces or $\texttt{RCD}^*(K,N)$ spaces (\cite{sturm1},\cite{sturm2},\cite{villani},\cite{local},\cite{R}) are both synthetic counterparts of Riemannian manifolds with Ricci curvature bounded below by $K$ and dimension bounded above by $N$.  Many classical results that hold for manifolds with Ricci curvature bounded below remain true in the synthetic setting. Here we mention some of them that will be used in the  present paper. In \cite{gigli}, Gigli established the splitting theorem on $\texttt{RCD}(0,N)$ spaces.  In \cite{uni1}, the existence of the universal cover of  an $\texttt{RCD}^*(K,N)$ space is established by Mondino-Wei. This universal cover admits a natural $\texttt{RCD}^*(K,N)$ structure such that the covering projection is a local metric measure isometry.  
Recently in \cite{uni2},  based on the results of \cite{uni1}, Wang proved that  an $\texttt{RCD}^*(K,N)$ space is semi-locally simply connected, and hence its universal cover is simply connected.  In \cite{nb}, Deng proved that an $\texttt{RCD}(K,N)$ space is nonbranching.

The main purpose of this paper is to generalize  the following result in \cite{Ye} to $\texttt{RCD}(0,N)$ spaces:

\begin{theorem}\label{recall} If $M$ is an open Riemannian n-manifold with $\text{Ric}_M\geq0$, then $b_1(M) \leq n-1$ and the equality holds if and only if $M$ is flat with a soul $T^{n-1}$. 
\end{theorem}

See \cite{Ye} for a background of Theorem \ref{recall}.

Throughtout this paper, for a $\texttt{CD}^*(K,N)$ space $(X,d,\mathfrak{m})$, we will always assume that $\text{supp}\mathfrak{m}=X$. This is not a restrictive condition since if $(X,d,\mathfrak{m})$ is  $\texttt{CD}^*(K,N)$  then $(\text{supp}\mathfrak{m},d,\mathfrak{m})$ is also $\texttt{CD}^*(K,N)$.

The following is  the main result of this paper:
\begin{theorem}\label{rigidity} 
Let $N\geq2$ and let $(X,d,\mathfrak{m})$ be a noncompact $\texttt{RCD}(0,N)$ space. Then the first Betti number $b_1(X)\leq [N]-1$. If $N\in \mathbb{N}_+$, then the equality holds if and only if $(X,d,\mathfrak{m})$ is  either a flat $N$-manifold with a soul $T^{N-1}$ or $[0,\infty)\times T^{N-1}$, both with $\mathfrak{m}$ a multiple of the Lebesgue measure.
\end{theorem}
\begin{remark}
	The above result is optimal in the sense that when $N\notin\mathbb{N}_+$,  the maximal first Betti number rigidity in the metric measure sense fails. Indeed, for each $h>0$, the space $([0,\infty),d,V_h\mathcal{L}^1)$ with $V_h(x)=f^h(x)$ and $f(x):[0,\infty)\rightarrow (0,2)$ convex  is $\texttt{RCD}(0,h+1)$, where $d$ is the Euclidean distance and $\mathcal{L}^1$ is the Lebesgue measure (cf. \cite{ent} Proposition 3.21). So the product of $([0,\infty),d,V_h\mathcal{L}^1)$ and $T^{[N]-1}$ with the Euclidean disdance and  the Lebesgue measure is an $\texttt{RCD}(0,N)$ space for every $0<h<N-[N]$ (see \cite{local} Theorem 4.1 and \cite{tensor}Theorem 7.6 for related tensorlization properties of $\texttt{RCD}(0,N)$ spaces).
	
The author conjectures  that the metric rigidity still holds in this case, i.e. that $(X,d)$ is isomorphic to a flat $N$-manifold with a soul $T^{N-1}$ or $[0,\infty)\times T^{N-1}$ if $(X,d,\mathfrak{m})$ is a noncompact $\texttt{RCD}(0,N)$ space with $b_1(X)=[N]-1$.     
\end{remark}

One knows from \cite{uni1} and \cite{uni2} that for a compact $\texttt{RCD}(0,N)$ space $(M,d,\mathfrak{m})$, the first Betti number $b_1(M) \leq [N]$, with equality holds if and only if $(M,d,\mathfrak{m})$ is a $T^{[N]}$ with $\mathfrak{m}$ a multiple of the Lebesgue measure.
This generalized a classical theorem of Bochner.
 For further results related to the first Betti number of compact $\texttt{RCD}$ spaces, see \cite{upper}.

As in \cite{Ye}, our approach to Theorem \ref{rigidity} is the estimate of the orbit growth of covering group actions.

\begin{definition}\label{def1}(\cite{Ye})
	Denote by $\#(A)$   the number of elements in a set $A$.		 Let $(X,d)$ be a metric space and let $\text{Isom}(X)$  be its  isometry group. Let $\Gamma$ be a subgroup of $\text{Isom}(X)$. For every $ x\in X$, put $	D^{\Gamma}(x,r)=\{g\in\Gamma :d(x,g(x))\leq r\}$.
	
	Given $p\in [0,\infty)$, we say  $\Gamma$ has polynomial orbit growth related  to $x$ of order $\geq p \,\,(\leq p)$, if and only if 
	\begin{equation*}
		\liminf\limits_{r\to\infty} \frac{\#(D^{\Gamma}(x,r))}{r^p}>0\,\,\,(	\limsup\limits_{r\to\infty} \frac{\#(D^{\Gamma}(x,r))}{r^p}<\infty).
	\end{equation*}
	We say  $\Gamma$ has polynomial orbit growth related  to $x$
	of order $> p\,\,(<p)$, if and only if 
	\begin{equation*}
		\lim_{r\to\infty} \frac{\#(D^{\Gamma}(x,r))}{r^p}=\infty\,\,(=0).
	\end{equation*}
\end{definition}
 
It is easy to check that the polynomial orbit growth properties defined above do not depend on the choice of the base point $x$. 

 Similar to \cite{Ye}, Theorem \ref{rigidity} follows easily from the following theorem on orbit growth  (cf. Proof of Theorem 1 in \cite{Ye}):
\begin{theorem}\label{orbit1}
Let $N\geq2$ and let $(X,d,\mathfrak{m})$ be a noncompact  $\texttt{RCD}(0,N)$ space. Let $\pi:(\tilde{X},\tilde{p})\rightarrow (X,p)$ be the universal cover with deck transformation group $\Gamma$, where $p\in X$. 

If $N$ is an integer, then $\Gamma$ has polynomial orbit growth of order $\leq N-1$. Moreover, $\Gamma$ fails to have polynomial orbit growth of order $<N-1$, i.e.  there exist $r_i\to\infty$ such that
\begin{equation*}
\lim_{i\to\infty} \frac{\#(D^{\Gamma}(\tilde{p},r_i))}{r_i^{N-1}}>0
\end{equation*}
 if and only if $(X,d,\mathfrak{m})$ is a flat $N$-manifold with an $N-1$ dimensional soul or $[0,\infty)\times X_1$ with $X_1$ is a compact flat
$(N-1)$-manifold, both with $\mathfrak{m}$ a multiple of the Riemannian volumes.

If $N$ is not an integer, then $\Gamma$ has polynomial orbit growth of order $<N-1$. 
\end{theorem}

As in \cite{split} and \cite{Ye}, the key  to proving Theorem \ref{orbit1} is to control the volume growth of the set of points that are close to the cut locus of a base point in some sense (see Lemma \ref{key} (2)). To obtain Lemma \ref{key} (2), we will use the Brunn-Minkowski inequality to replace the volume element comparison  argument used in \cite{Ye}. The nonbranching property of $\texttt{RCD}(K,N)$ spaces is also crucial for our proof.  

We also generalize Theorem 4 in \cite{Ye} (which is the orbit version of Theorem 1.1 in \cite{anderson} )
 to nonbranching $\texttt{CD}^*(K,N)$ spaces:
 
 \begin{theorem}\label{orbit2} Let $(X,d,\mathfrak{m})$ be a nonbranching $\texttt{CD}^*(K,N)$ space.
 	
 (1) Let $p,q\in X,p\neq q$. Put $E(p,q)=\{x=X\mid d(p,x)=d(q,x) \}$. Then $\mathfrak{m}(E(p,q))=0$. 
 
 (2)Let $\pi:\left(\bar{X},\bar{d}, \bar{\mathfrak{m}},\bar{x}_0\right) \rightarrow \left(X,d,\mathfrak{m},x_0\right)$ be a normal covering with deck transformation group $G$ (here $\left(\bar{X},\bar{d}, \bar{\mathfrak{m}}\right)$ is the lift of $\left(X,d,\mathfrak{m}\right)$, cf. 7.2 of \cite{local} or 2.2 of \cite{uni1}). Then for every $r>0$ we have:
 
 	\begin{align}
 	\label{my}	\#(D^G(\bar{x}_0,2r))\cdot \mathfrak{m}(B_r(x_0))&\geq \bar{\mathfrak{m}}(B_r(\bar{x}_0)),\\
 	\label{and}	\#(D^G(\bar{x}_0,r))\cdot \mathfrak{m}(B_r(x_0))&\leq \bar{\mathfrak{m}}(B_{2r}(\bar{x}_0)).
 \end{align}	
 \end{theorem}

It has been well known since \cite{yau2} that an open manifold with nonnegative Ricci curvature has at least linear volume growth (i.e. $\liminf\limits_{R\to\infty}\frac{\text{Vol}(B_R(p))}{R}>0$). This fact also holds for $\texttt{RCD}(0,N)$ spaces (cf. \cite{linear}). So the following definition makes sense:
\begin{definition}
Let $(X,d,\mathfrak{m})$ be a noncompact $\texttt{RCD}(0,N)$ space. We say $X$ has minimal volume growth if and only if 
\begin{equation*}
\limsup\limits_{R\to\infty}\frac{\mathfrak{m}(B_R(p))}{R}<\infty \text{ for some $p\in X.$}
\end{equation*}
\end{definition}

Let $(X,d,\mathfrak{m})$ be a $\texttt{CD}(0,N)$ space.  By the Bishop-Gromov inequality, the limit $\lim\limits_{r\to\infty}\frac{\mathfrak{m}(B_r(p))}{r^N}$  always exists and does not rely on the choice of $p\in X$. We say that $X$ has Euclidean volume growth if $\lim\limits_{r\to\infty}\frac{\mathfrak{m}(B_r(p))}{r^N}>0$. Otherwise, we say that $X$ collapses at infinity.

The following theorem, which generalized Theorem 2 of \cite{Ye}, is an immediate consequence of  Theorem \ref{orbit1} and Theorem \ref{orbit2} (2) (\ref{my}):
\begin{theorem}\label{volume}
Let $(X,d,\mathfrak{m})$ be a noncompact $\texttt{RCD}(0,N)$ space and let $p\in X$. Let $\pi:(\tilde{X},\tilde{d},\tilde{\mathfrak{m}})\rightarrow (X,d,\mathfrak{m})$ be the universal cover. If $X$ has minimal volume growth, then:

\noindent if $N\notin\mathbb{N}_+$, then $\tilde{X}$ collapses at infinity;

\noindent  if $N\in\mathbb{N}_+$, then $\tilde{X}$ has Euclidean volume growth  if and only if $(X,d)$ is a flat $N$-manifold with an $N-1$ dimensional soul or $[0,\infty)\times X_1$, where $X_1$ is a compact flat
$N-1$ manifold.
\end{theorem}

Theorem \ref{volume} can be viewed as a generalization of  the rigidity result of Theorem \ref{rigidity}. Indeed, by Theorem \ref{orbit2}, when $N\in\mathbb{N}_+$, $b_1(X)=N-1$ implies that $X$ has minimal volume growth and that $\tilde{X}$ has Euclidean volume growth.

\section{Preliminaries}
In this section, we list the properties of $\texttt{CD}^*(K,N)$ spaces and $\texttt{RCD}^*(K,N)$ spaces that will be used in this paper.   We will not give  the definitions here, since they will  not be used explicitly. 

\paragraph{1}When $K=0$, the $\texttt{CD}^*(0,N)$
condition is the same as the $\texttt{CD}(0,N)$ condition.

\paragraph{2}A $\texttt{CD}^*(K,N)$ space $(X,d,\mathfrak{m})$  is a proper geodesic space (cf. \cite{BM} Lemma 3.5).  
\paragraph{3} An $\texttt{RCD}^*(K,N)$
space   is an infinitesimally Hilbertian  $\texttt{CD}^*(K,N)$ space (cf. \cite{uni1} and related reference there for the definition of the infinitesimally Hilbertian condition). Especially, an  $\texttt{RCD}^*(K,N)$
space is $\texttt{CD}^*(K,N)$.
\paragraph{4}For $K\in\mathbb{R}, N\geq1,0\leq t\leq 1,$ and $\theta\in\mathbb{R}_+$, set
$$\sigma_{K, N}^{(t)}(\theta):= \begin{cases}\infty & \text { if } K \theta^2 \geq N \pi^2, \\ 
\frac{\sin(\sqrt{\frac{K}{N}}t\theta)}{\sin(\sqrt{\frac{K}{N}}\theta)}	& \text { if } K>0,K \theta^2 <N \pi^2,\\
t& \text{ if } K=0,\\
	\frac{\sinh(\sqrt{-\frac{K}{N}}t\theta)}{\sinh(\sqrt{-\frac{K}{N}}\theta)} & \text { if }  K<0. \end{cases}$$

We will use the following Brunn-Minkowski inequality to prove Lemma \ref{key}:
\begin{theorem}
(\cite{local}Proposition 6.1) Let $K,N\in\mathbb{R}$ and let $N\geq 1$. Assume that $(X,d,\mathfrak{m})$ is a $\texttt{CD}^*(K,N)$ space. Then for all Borel sets $A,B\subset M$ and $t\in[0,1]$, 
\begin{equation*}
\mathfrak{m}(Z_t(A,B))^{\frac{1}{N}}\geq  \sigma^{(1-t)}_{K,N}(\Theta)\cdot \mathfrak{m}(A)^{\frac{1}{N}}+\sigma_{K,N}^{(t)}\cdot \mathfrak{m}(B)^{\frac{1}{N}},
\end{equation*}
where $Z_{t}(A,B)=\{x\in X\mid \text{there exist } a\in A \text{ and } b\in B \text{ such that  }  d(a,x)=td(a,b) \text{ and } d(b,x)=(1-t)d(a,b).    \}$ and where 
$$\Theta:= \begin{cases}\inf _{x_0 \in A, x_1 \in B} \mathrm{~d}\left(x_0, x_1\right), & K \geqslant 0 ,\\ \sup _{x_0 \in A, x_1 \in B} \mathrm{~d}\left(x_0, x_1\right), & K<0.\end{cases}$$
\end{theorem}     

\paragraph{5}We also need the splitting theorem in the nonsmooth setting, which generalized the celebrated Cheeger-Gromoll Splitting Theorem  \cite{split}:
\begin{theorem}\label{splitting}
	 (\cite{gigli})Let $(X,d,m)$ be an \texttt{RCD}$(0,N)$ space with $1\leq N<\infty$. Suppose that $X$ contains a line. Then $(X,d,m)$ is isomorphic to $(X^\prime\times\mathbb{R}, d^\prime\times d_{E},m^\prime\times \mathcal{L}^1)$, where $d_E$ is the Euclidean distance, $\mathcal{L}^1$ the Lebesgue measure and $(X^\prime,d^\prime,m^\prime)$ is an \texttt{RCD}$(0,N-1)$ space if $N\geq2$ and a singleton  if $N<2$.
\end{theorem}

\paragraph{6}We require the existence  of a simply connected universal cover: 
\begin{theorem}(\cite{uni1},\cite{uni2})\label{univer}
	Let $(X,d,m)$ be an $\texttt{RCD}^*(K,N)$-space for some  $K\in\mathbb{R},1< N<\infty$. Then $(X,d,m)$ admits a simply connected universal cover $(\tilde{X},\tilde{d},\tilde{m})$ which is itself an $\texttt{RCD}^*(K,N)$-space.  
\end{theorem}

\paragraph{7} We say that a geodesic metric space $(X,d)$ is nonbranching if and  only if there are no 4  different points $x,y,z,w$ in $X$ such that $d(x,z)=d(x,y)+d(y,z)$ and $d(x,w)=d(x,y)+d(y,w)$.

\begin{theorem}(\cite{nb}) Let $(X,d,\mathfrak{m})$ be an $\texttt{RCD}(K,N)$ space. Then $(X,d,\mathfrak{m})$ is nonbranching.
\end{theorem}

\paragraph{8}Finally, we will use the following splitting result for metric isometry groups. It was used in \cite{split} under the condition that $N$ is a Riemannian manifold. 

\begin{prop}\label{iso}
Let $(N,d)$ be a metric space which contains no line and let $k\in\mathbb{N}_+$. Then any isometry $F$ of the metric product $N\times \mathbb{R}^k$ splits as $F=(f,g)$, where $f$ is an isometry on $N$ and $g$ is an isometry on $\mathbb{R}^k$.
\end{prop}
We will give a proof of Proposition \ref{iso} in the Appendix.

\section{Proofs}
In this section, we will first prove Theorem \ref{orbit2} and then prove Theorem \ref{orbit1}, since our proof of Theorem \ref{orbit1} slightly uses Theorem \ref{orbit2}.

Let $(X,d)$ be a proper geodesic space and let $p\in X$. Put
\begin{align*}
C_r(p)=&\{q\in X\mid \forall z\in X\backslash B_r(q)  , d(p,q)+d(q,z)>d(p,z)   \}.
\end{align*}
Then $C_r(p)$ is open in X. The set $C(p):=\bigcap\limits_{i=1}^\infty C_{i^{-1}}(p)$ is called the cut locus of $p$. 

 Let $I=[0,l]$ or $[0,\infty)$, where $0<l<\infty$. We say that a geodesic  $\gamma:I\rightarrow X$ with $\gamma(0)=p$ is non-extendable (relative to $p$) if and only if either $I=[0,l]$ and $\gamma(l)\in C(p)$ or $I=[0,\infty)$, i.e $\gamma$ is a ray. The set of all non-extendable geodesics (relative to $p$) is denoted by $\text{NE}(p)$. 

The following measure estimate is the key to the  proofs of our Theorems.
\begin{lemma}\label{key}
Let $N>1$ and let $(X,d,\mathfrak{m})$ be a nonbranching $\texttt{CD}^*(K,N)$ space. Let $p\in X$.

(1) If $A\subset X$ is a Borel set such that $A\cap \gamma$ contains at most one point for  every $\gamma\in\text{NE}(p)$.  Then $\mathfrak{m}(A)=0$.

(2)If $K=0$, then for every $r>0$ we have
\begin{equation}\label{eq0}
\lim\limits_{R\to\infty}\frac{\mathfrak{m}(C_r(p)\cap B_R(p))}{R^{N-1}}=0 .
\end{equation} 
\end{lemma}

\begin{proof}
(1)Without loss of generality, we may assume $K=-1$. For every $i,R\in\mathbb{N}_+$, put 
\begin{align*}
A(R)&=A\cap (B_{R}(p)\backslash B_{R^{-1}}(p)),\\
 A_{ij}(R)&=A\cap\big(B_{R^{-1}+\frac{j+1}{i}}(p)\backslash B_{R^{-1}+\frac{j}{i}}(p)\big),j=0,1,\cdots,Ri-1.
\end{align*}
By using the Brunn-Minkowski inequality for $A_{ij}(R),\{p\}$ and time $t_{ij}=1-\frac{(2R)^{-1}}{R^{-1}+\frac{j+1}{i}} $, we obtain:
\begin{equation*}
\mathfrak{m}(Z_{t_{ij}}(A_{ij}(R),\{p\}))\geq c(N,R) \mathfrak{m}(A_{ij}(R))
\end{equation*}
where $c(N,R)>0$. Note that  the sets $Z_{t_{ij}}(A_{ij}(R),\{p\})$ are disjoint by the nonbranching condition and by the nature of $A$. Note also that the set $Z_{t_{ij}}(A_{ij}(R),\{p\})$ is contained in $ B_{(2R)^{-1}}(p)\backslash B_{(2R)^{-1}-\frac{1}{i}}(p)$ for every $j$. Now, if every $Z_{t_{ij}}(A_{ij}(R),\{p\})$   is a Borel set, we have
\begin{align*}
\mathfrak{m}(A(R))&=\sum_{j=0}^{Ri-1} \mathfrak{m}(A_{ij}(R))\\
&\leq (c(N,R))^{-1}\sum_{j=0}^{Ri-1} \mathfrak{m}(Z_{t_{ij}}(A_{ij}(R),\{p\}))\\
&\leq (c(N,R))^{-1}\mathfrak{m}(B_{(2R)^{-1}}(p)\backslash B_{(2R)^{-1}-\frac{1}{i}}(p))    .   
\end{align*}
Let $i\to\infty$, we get $\mathfrak{m}(A(R))=0$.
Since $A\backslash\{p\} =\bigcup\limits_{R=1}^\infty A(R) $, we obtain $\mathfrak{m}(A)=0$.

In general, $Z_{t_{ij}}(A_{ij}(R),\{p\})$ may not  be measurable. In this case, we  use the inner regularity of $\mathfrak{m}$ (note that $\mathfrak{m}$ is a Radon measure, cf. Theorem 7.8 in \cite{real} ). For every $\epsilon>0$ we can find compact $K_{ij}\subset A_{ij}(R)$ such that $\sum_{j=0}^{Ri-1}\mathfrak{m}(A_{ij(R)}\backslash K_{ij})<\epsilon$. Note that $Z_{t_{ij}}(K_{ij},\{p\})$ are also compact. So 
\begin{align*}
\mathfrak{m}(A(R))&< \sum_{j=0}^{Ri-1} \mathfrak{m}(K_{ij})+\epsilon\\
&\leq (c(N,R))^{-1}\sum_{j=0}^{Ri-1} \mathfrak{m}(Z_{t_{ij}}(K_{ij},\{p\}))+\epsilon\\
&\leq (c(N,R))^{-1}\mathfrak{m}(B_{(2R)^{-1}}(p)\backslash B_{(2R)^{-1}-\frac{1}{i}}(p))   +\epsilon .   
\end{align*}
Since $i$ and $\epsilon$ are arbitary, we conclude that $\mathfrak{m}(A(R))=0$.
 	
(2)We use the following sublemma:
\begin{sublemma}\label{sub}
Let $R>1$ and let $A\subset B_{R+1}(p)\backslash B_R(p)$ be  compact. Denote by $K_A$ the union of all geodesics connecting $p$ and a point $a\in A$. Put $V_A=B_1(p)\cap K_A$. then $\mathfrak{m}(A)\leq  N(R+1)^{N-1}\mathfrak{m}(V_A)$.	
\end{sublemma}
Assume the Sublemma \ref{sub} holds. Without loss of generality, we may assume $r=1$.
Put $C_i=C_1(p)\cap (B_{i+1}(p)\backslash B_i(p))$ for $i=2,3,\cdots$. By the inner regularity of $\mathfrak{m}$, We may choose compact $C_i^\prime\subset C_i$ such that $\sum_{i=2}^{\infty} \mathfrak{m}(C_i\backslash C_i^\prime)<1$. Note that by the nonbranching condition and the definition  of $C_1(p	)$, $V_{C_i^\prime}\cap V_{C_j^\prime}=\{p\}$ for $|i-j|\geq 2$.  The compactness of $C_i^\prime$ guarantee that $K_{C_i^\prime}$ are compact, hence $V_{C_i^\prime}$ are Borel. So $\sum_{i=2}^{\infty} \mathfrak{m}(V_{C_i^\prime})\leq 2\mathfrak{m}(B_1(p))$.
By Sublemma \ref{sub}, we have
\begin{align*}
\mathfrak{m}(C_1(p)\cap B_R(p))\leq& \mathfrak{m}(B_2(p))+\sum_{i=2}^{[R]}\mathfrak{m}(C_i) \\
<& \mathfrak{m}(B_2(p))+\sum_{i=2}^{[R]}\mathfrak{m}(C_i^\prime)+1\\ 
 \leq&  \mathfrak{m}(B_2(p))+\sum_{i=2}^{[R]} N(i+1)^{N-1}\mathfrak{m}(V_{C_i^\prime})+1\\
=& \mathfrak{m}(B_2(p))+ \sum_{i=2}^{[\sqrt{R}]} N(i+1)^{N-1}\mathfrak{m}(V_{C_i^\prime})+ \sum_{i=[\sqrt{R}]+1}^{[R]} N(i+1)^{N-1}\mathfrak{m}(V_{C_i^\prime})+1\\
\leq& \mathfrak{m}(B_2(p))+2N([\sqrt{R}]+1)^{N-1}\mathfrak{m}(B_1(p))\\
&+N([R]+1)^{N-1}\sum_{i=[\sqrt{R}]+1}^{[R]}\mathfrak{m}(V_{C_i^\prime})+1.
\end{align*}
Since $\lim\limits_{R\to\infty}\sum_{i=[\sqrt{R}]+1}^{[R]}\mathfrak{m}(V_{C_i^\prime})=0$, we obtained (\ref{eq0}).

\begin{proof}[Proof of the Sublemma \ref{sub}]
Note that $K_A$ is compact, hence $V_A$ is Borel. By using the Brunn-Minkowski inequality for $A,\{p\}$ and $t=1-\frac{1}{R+1}$, we get:
\begin{equation}\label{eq1}
\mathfrak{m}(Z_{t}(A,\{p\}))\geq  (1-t)^N\mathfrak{m}(A).
\end{equation}
Put $F=V_A\backslash B_{\frac{R}{R+1}}(p)$, then $Z_t(A,\{p\})\subset F$. Use the Brunn-Minkowski inequality for $V_A,\{p\}$ and $t^\prime=\frac{1}{R+1}$, we get:
\begin{equation}\label{eq2}
\mathfrak{m}(Z_{t^\prime}(V_A,\{p\}))\geq (1-t^\prime)^N \mathfrak{m}(V_A)
\end{equation}
Note that $Z_{t^\prime}(V_A,\{p\})=V_A\backslash F$, so we obtain from (\ref{eq2}) that:
\begin{equation}\label{eq3}
\mathfrak{m}(F)\leq \left(1-\big(1-\frac{1}{R+1}\big)^N \right)\mathfrak{m}(V_A)\leq N(R+1)^{-1}\mathfrak{m}(V_A).
\end{equation}
Now (\ref{eq1}) and (\ref{eq3}) together give:
\begin{equation*}
	\mathfrak{m}(A)\leq (R-1)^N\mathfrak{m}(F)\leq N(R+1)^{N-1}\mathfrak{m}(V_A).
\end{equation*}  
\end{proof}

\end{proof}
\begin{remark}
Let $(X,d,\mathfrak{m})$ be a nonbranching $\text{CD}^*(K,N)$ space.We obtain immediately from Lemma \ref{key} (1) that the cut locus of a point $p\in X$ has measure 0.
\end{remark}

\begin{proof}[Proof of Theorem \ref{orbit2}] (1) For every  $\gamma\in \text{NE}(p)$, let us  prove that $E(p,q)\cap \gamma$ contains at most one point. This implies   $\mathfrak{m}(\text{E}(p,q))=0$ by Lemma \ref{key} (1) (note that $E(p,q)$ is closed, thus Borel).
	
	Assume that there are two points $a,b$ in $\text{E}(p,q)\cap\gamma$ and $d(p,a)<d(p,b)$.  Then
	\begin{align*}
		d(b,q)&= d(b,p)\\
		&=d(p,a)+d(a,b)\\
		&=d(q,a)+d(a,b).
	\end{align*}
	This causes the branching of geodesics.

	(2) Note that $(\bar{X},\bar{d},\bar{\mathfrak{m}})$ is still a nonbranching $\text{CD}^*(K,N)$ space by \cite{local}, so the result of (1) is applicable to $(\bar{X},\bar{d},\bar{\mathfrak{m}})$. For every $g\in G$, put $D_g=\{z\in \bar{X}\mid d(z,\bar{x}_0)<d(z,g\bar{x}_0)\}$. The nonbranching condition implies that $$\partial D_g=\{z\in \bar{X}\mid d(z,\bar{x}_0)=d(z,g\bar{x}_0)\}.$$ 
By (1), $\mathfrak{m}(\partial D_g)=0$ for every $g\in G$. 	Define the Dirichlet domain $F$ associated to $\bar{x}_0$ by $F=\bigcap\limits_{g\in G,g\neq e} D_g$. Denote by $\overline{F}$ the closure of $F$ in $\bar{X}$. Since $\partial F= \overline{F}\backslash F$ is contained in the countable union of $\partial D_g$, we obtain $\mathfrak{m}(\partial F)=0$. The proof of (2) then follows word by word from the proof of Theorem 1.1 of \cite{anderson} and Theorem 4 of \cite{Ye}.

\end{proof}

Note that an $\texttt{RCD}(K,N)$ space is nonbranching by \cite{nb}. So  Lemma \ref{key} (2) is applicable to $\texttt{RCD}(0,N)$ spaces. 
\begin{proof}[Proof of Theorem \ref{orbit1}] 
By the splitting theorem, we have $$(\tilde{X},d_{\tilde{X}},\tilde{\mathfrak{m}},\tilde{x}_0)\cong (Y\times\mathbb{R}^k,d_Y\times d_{\mathbb{R}^k},\mathfrak{m}_Y\times \mathcal{L}^k,(y_0,0^k)),$$ where $k\leq [N]$ is an integer, $\mathcal{L}^k$ is the Lebesgue measure and $(Y,d_Y,\mathfrak{m}_Y)$ is a simply connected  $\texttt{RCD}(0,N-k)$ space which contains no line. There are following three possibilities:

Case 1.  $k=[N]$, then $Y$ is a point. So $(X,d,\mathfrak{m})$ is a flat $[N]$-manifold with $\mathfrak{m}$ a multiple of the Riemannian volume. In this case, $\Gamma$ has polynomial orbit growth of order $\leq k-1=[N]-1$, and $\Gamma$ fails to have polynomial orbit growth of order $<[N]-1$ if and only if $X$ has a $[N]-1$ dimensional soul (cf. Proposition 4 (1) of \cite{Ye}).

Case 2. $k=[N]-1$, then $(Y,d_Y,\mathfrak{m}_Y)$ is an $\texttt{RCD}(0,N-[N]+1)$ space. Since $N-[N]+1<2$,
 we have  by Corollary 1.2 in  \cite{low} that  $(Y,d_Y)=[0,T]$ for some $T\geq0$ or $(Y,d_Y)=[0,\infty)$ since $Y$ is simply connected and $Y\neq \mathbb{R}$. 
 
 If $(Y,d_Y)=[0,T]$, then one may use Theorem  \ref{orbit2} (2) (\ref{and}) and the fact that $(X,d,\mathfrak{m})$ has at least linear volume growth to conclude that $\Gamma$ has polynomial orbit growth of order $\leq [N]-2$. 
 
 If $(Y,d_Y)=[0,\infty)$, then  $\text{Isom}_{met}(Y)$,  the metric isometry group of $Y$, is trivial. So $\Gamma\subset \text{Isom}_{met}(\mathbb{R}^k)$.  In the case that $\mathbb{R}^k/\Gamma$ is a compact manifold,  $\Gamma$ has a finite index normal subgroup $\mathbb{Z}^k$ (the translation part ) by Bieberbach theorem. So $\Gamma$ has polynomial orbit growth of order $\geq k$ and $\leq k$ (cf. Lemma 3 of \cite{Ye}).  In the case that  $\mathbb{R}^k/\Gamma$ is an open manifold, 
 $\Gamma$ has polynomial orbit growth of order $\leq k-1=[N]-2$ (cf. Proposition 4 (1) of \cite{Ye}).

Case 3. $k\leq [N]-2$. Then $(Y,d_Y,\mathfrak{m}_Y)$ is an $\texttt{RCD}(0,N-k)$ space with $N-k\geq 2$, so Lemma \ref{key} (2) is applicable to $Y$. Fix an $l>0$ such that $B_l(g_1\cdot\tilde{x}_0)\cap B_l(g_2\cdot\tilde{x}_0)=\emptyset, \forall g_1,g_2\in\Gamma,g_1\neq g_2$. Since $Y$ contains no line, by contradiction argument there exists an $h>0$ such that $$\text{Isom}_{met}(Y)\cdot B_l(y_0)\subset B_h(y_0)\cup C_h(y_0),$$
 where $\text{Isom}_{met}(Y)$ is the metric isometry group of $Y$ (cf. Lemma 1 of \cite{Ye}).

Since the metric isometry group of $\tilde{X}$ splits, we have
\begin{equation*}
\bigcup\limits_{g\in\Gamma, d_{\tilde{X}}(\tilde{x}_0,g\cdot\tilde{x}_0)\leq R}B_{l}(g\cdot \tilde{x}_0)\subset \left( (B_h(y_0)\cup C_h(y_0))\cap B_{R+l}(y_0)\right) \times B_{R+l}(0^k),\forall R>0.
\end{equation*}
 So 
\begin{align*}
\tilde{\mathfrak{m}}(B_l(\tilde{x}_0))\cdot \#(D^\Gamma_R(\tilde{x}_0))&\leq \mathfrak{m}_Y( (B_h(y_0)\cup C_h(y_0))\cap B_{R+l}(y_0))\cdot \mathcal{L}^k(B_{R+l}(0^k))\\
&\leq \big(\mathfrak{m}_Y(B_h(y_0))+\mathfrak{m}_Y(C_h(y_0)\cap B_{R+l}(y_0)) \big)\omega_k(R+l)^k\\
&=\omega_k\mathfrak{m}_Y(B_h(y_0))(R+l)^k+\omega_kf(R+l)(R+l)^{N-1}. \,\,(*)
\end{align*} 
 where $\omega_k=\mathcal{L}^k(B_1(0^k))$ and
 $\mathfrak{m}_Y(C_h(y_0)\cap B_{R+l}(y_0))=f(R+l)(R+l)^{N-k-1}$ and we have $\lim\limits_{r\to\infty} f(r)=0$ by Lemma \ref{key} (2). 
 So  $\Gamma$ has polynomial orbit growth of order $< N-1$ in this case. 
 \\
 
 Now if $N\notin\mathbb{N}$, we conclude that $\Gamma$ has polynomial orbit of order $<N-1$. 
 
 If $N\in\mathbb{N}$, we conclude that $\Gamma$ has polynomial orbit of order $\leq N-1$.  Now consider the situation  that  $\Gamma$ fails to have polynomial orbit of order $< N-1$. Then $k$ can only equals $N$ or $N-1$. If $k=N$, then $(X,d,\mathfrak{m})$ is a flat $N$-manifold with an $N-1$ dimensional soul and with $\mathfrak{m}$ a multiple of the Riemannian volume. If $k=N-1$, then $(Y,d_Y)=[0,\infty)$ and $\mathbb{R}^k/\Gamma$ is a compact flat $(N-1)$-manifold as analysed in Case 2. Since  $(Y,d_Y,\mathfrak{m}_Y)$ is an $\texttt{RCD}(0,1)$ space, $\mathfrak{m}_Y$ can only be a multiple of $\mathcal{L}^1$ (see the Remark \ref{1} below). This completes the proof.

\end{proof}
\begin{remark} \label{1}
If $(Y,d_Y)=[0,\infty)$ and $(Y,d_Y,\mathfrak{m}_Y)$ is an $\texttt{RCD}(0,1)$ space, we show that $\mathfrak{m}_Y$ can only be a multiple of $\mathcal{L}^1$ here: For $i\in\mathbb{N}_+$, let $I_j=[\frac{j}{i},\frac{j+1}{i}],j=0,1,2,\cdots$. Use the Brunn-Minkowski inequality for $A=I_{j}\cup I_{j+1},B=\{\frac{j}{i}\}$ and $t=\frac{1}{2}$, we obtain that $\mathfrak{m}_Y(I_j)\geq  \frac{1}{2}\mathfrak{m}_Y(A)$. Simialrly, $\mathfrak{m}_Y(I_{j+1})\geq  \frac{1}{2}\mathfrak{m}_Y(A)$. Since a single point has measure $0$, we have $\mathfrak{m}_Y(I_j)=\mathfrak{m}_Y(I_{j+1})$. Since  $j$ is arbitrary, we obtain that  $\mathfrak{m}_Y([p,q])=c(q-p),\forall 0\leq p\leq q,p,q\in\mathbb{Q}$, where $c=\mathfrak{m}_Y([0,1])>0$. We conclude that $\mathfrak{m}_Y([a,b])=c(b-a),\forall 0\leq a\leq b,a,b\in\mathbb{R}$ by taking limits. Since every open set of $[0,\infty)$ is a disjoint countable union of intervals, we conclude that $\mathfrak{m}_Y=c\mathcal{L}^1$ on open sets. It   follows from a classical result in measure theory that $\mathfrak{m}_Y=c\mathcal{L}^1$ on all Borel sets.  
\end{remark}

\section{Appendix:  A proof of Proposition \ref{iso}}

Fix $(p,v)\in N\times \mathbb{R}^k$, and set $F(p,v)=(q,w)$. Then $F(\{p\}\times \mathbb{R}^k)=\{q\}\times \mathbb{R}^k$ since $F$ maps a line to another line and since $N$ contains no line.  
\begin{claim}\label{cliso}
$F(N\times \{v\})=N\times\{w\}$.
\end{claim}
\begin{proof}
For $n\in N$,  assume that $F(n,v)=(m,w^\prime)$. Note that $(p,v)$ is the unique point on $\{p\}\times\mathbb{R}^k$ that is closest to $(n,v)$. So $F(p,v)=(q,w)$ is the unique point on $\{q\}\times\mathbb{R}^k$ that is closest to $F(n,v)$. Since $(q,w^\prime)$ is the unique point on $\{q\}\times\mathbb{R}^k$ that is closest to $F(n,v)=(m,w^\prime)$, we obtain that $w=w^\prime$. So $F(N\times\{v\}\subset N\times\{w\})$. The same argument shows that $F^{-1}(N\times\{w\})\subset N\times\{v\}$. This proves the Claim.

\end{proof}
By Claim \ref{cliso},  there are isometry  $f:N\rightarrow N$
and $g:\mathbb{R}^k\rightarrow \mathbb{R}^k$ such that $F(x,v)=(f(x),w),F(p,y)=(q,g(y)),\forall x\in N,y\in\mathbb{R}^k$.

Now we check that $F(x,y)=(f(x),g(y))$. To see this, note that $(x,v)$ is the unique point on $N\times\{v\}$ that is closest to $(x,y)$, and that $(p,y)$ is the unique point on $\{p\}\times\mathbb{R}^k$ that is closest to $(x,y)$.
 So $F(x,v)=(f(x),w)$ is the unique point on $N\times \{w\}$ that is closest to $F(x,y)$, and $F(p,y)=(q,g(y))$ is the unique point on $\{q\}\times\mathbb{R}^k$ that is closest to $F(x,y)$. That is $F(x,y)= (f(x),g(y))$.

\section{Acknowledgement}
The author thanks his advisor Professor Xiaochun Rong for suggesting  this problem and for helpful discussion.

\clearpage
\bibliography{refs}
\end{document}